\documentclass[12pt, a4paper]{article}
\usepackage[utf8]{inputenc}
\usepackage[T1]{fontenc}
\usepackage{amsmath, amsthm, amssymb}
\usepackage{geometry}
\usepackage{enumitem}
\usepackage{graphicx}
\usepackage{hyperref}
\usepackage{MnSymbol,wasysym}
\usepackage{cleveref} 
\usepackage{comment}
\usepackage{tikz}
\usetikzlibrary{matrix}
\usetikzlibrary{arrows}
\newtheorem{theorem}{Theorem}[section]
\newtheorem{corollary}[theorem]{Corollary}
\newtheorem{lemma}[theorem]{Lemma}
\newtheorem{proposition}[theorem]{Proposition}

\theoremstyle{definition}
\newtheorem{definition}[theorem]{Definition}
\newtheorem{example}[theorem]{Example}

\theoremstyle{remark}

\title{Periodicity Uncovered: A Deep Dive into Bott's Theorems in K-Theory and Fiber Bundles}
\author{Ivan Z. Feng\thanks{\href{mailto:ifeng@usc.edu}{ifeng@usc.edu} \quad \textbar \quad \href{https://dornsife.usc.edu/ivan/}{dornsife.usc.edu/ivan}}}
\date{May 10, 2023}

\usepackage{fancyhdr}

\begin{document}
\maketitle

\begin{abstract}
This paper presents a comparison between two versions of Bott Periodicity Theorems: one in topological K-theory and the other in stable homotopy groups of classical groups. It begins with an introduction to K-theory, discussing vector bundles and their role in understanding the algebraic and topological aspects of these spaces. Then the two versions of Bott periodicity, as well as the topological notions necessary to understand them, are further explored. The aim is to illustrate the connections and distinctions between these two theorems, deepening our understanding of their underlying mathematical structures such as topological K-theory and fiber bundles.
\end{abstract}

\begin{center}
  \tableofcontents
\end{center}
\section{Introduction to Essentials of K-Theory}

In this section, we aim to introduce two versions of K-theory, $K(X)$ and $\widetilde{K}(X)$ for a fixed base space $X$, as well as their related motivations and properties. They are the foundation of further development of the Bott periodicity.

\subsection{Vector Bundles}

Let's begin with the foundation of topological K-theory: vector bundles. The concept of vector bundles originated in the early 20th century, with significant contributions from mathematicians such as Hermann Weyl, Hassler Whitney, and Norman Steenrod. It has since become a fundamental object of study in algebraic and differential topology, as well as a central tool in areas such as gauge theory and index theory.

The main idea behind vector bundles is to associate a vector space to every point of a topological space $X$ in such a way that these vector spaces fit together in a locally trivial manner. This structure allows us to study local properties of the base space $X$ by analyzing the corresponding properties of the associated vector spaces. With this motivation in mind, we can now introduce the formal definition of vector bundles:

\begin{definition}[Vector Bundle]\label{def:vector_bundle}  
A \textbf{vector bundle} is a family of vector spaces $p: E \rightarrow X$ that has a local strong basis at every point of $X$. In other words, for each $x \in X$ there is a neighborhood $x \in U \subseteq X$, an $n \in \mathbb{Z}_{\geq 0} \cup\{\infty\}$, and an isomorphism\footnote{In this paper we adopt the symbol $\approx$ to mean \textit{isomorphism} as used in Hatcher's \cite{hatcher2002algebraic} and \cite{hatcher2003vector}, instead of the more ubiquitous symbol  $\cong$ as used in \cite{atiyah1967k} and \cite{dugger2007notes}. In general, the symbol $\approx$ is usually used to denote a weaker notion of equivalence that may allow for more flexibility or ambiguity in the choice of isomorphisms, while $\cong$ is a stronger notion of equivalence that implies a more canonical choice of isomorphism.} of families of vector spaces
\begin{center}
\begin{tikzpicture}
\matrix (m) 
[matrix of math nodes,row sep=4em,column sep=4em,minimum width=2em] 
{
p^{-1}(U) & {U\times \mathbb{R} ^{n}} \\
          & U \\ 
}; 
\path[-stealth] 
(m-1-1) edge node [above] {$\approx$} (m-1-2)
        edge node [below] {} (m-2-2)
(m-1-2) edge [solid]  node [right] {} (m-2-2);
\end{tikzpicture}
\end{center}

\end{definition}

Note that when $n=\infty$ the space $\mathbb{R}^{n}$ has the colimit topology. The isomorphism in the above diagram is called a \textbf{local trivialization}. Usually, one simply says that a vector bundle is a family of vector spaces that is locally trivial. In particular, the \textbf{trivial bundle} (or product bundle) is given by $E=B \times \mathbb{R}^{n}$ with $p$ the projection onto the first factor. Let's provide some more typical examples.

\begin{example}[Line Bundle]\label{ex:line_bundle}
Let $E$ be the quotient space of $I \times \mathbb{R}$ under the identifications $(0, t) \sim(1,-t)$. The projection $I \times \mathbb{R} \rightarrow I$ induces a map $p: E \rightarrow S^{1}$, which is a 1-dimensional vector bundle. We call this bundle the \textbf{line bundle}. Here since $E$ is homeomorphic to a Möbius band with its boundary circle deleted, we also call this bundle the \textbf{Möbius bundle}.
\end{example}

\begin{example}[Canonical Line Bundle]
Consider the real projective $n$-space, denoted as $\mathbb{R}\mathrm{P}^n$, which represents the set of lines in $\mathbb{R}^{n+1}$ that pass through the origin. Since each such line intersects the  $n$-dimensional unit sphere $S^{n}$ in a pair of antipodal points, we can view $\mathbb{R}\mathrm{P}^n$ as the quotient space of $S^n$, where antipodal pairs of points are identified. Corresponding to this space, we have the \textbf{canonical line bundle} $p: E \rightarrow \mathbb{R}\mathrm{P}^n$. The total space of this bundle, $E$, is a subspace of $\mathbb{R}\mathrm{P}^n \times \mathbb{R}^{n+1}$, consisting of pairs $(\ell, v)$, where $v \in \ell$. The projection map $p$ is defined as $p(\ell, v) = \ell$. Local trivializations can be defined by orthogonal projection.
\end{example}

Now, to further examine this structure, let's proceed with the following definition.

\begin{definition}[$\operatorname{Vect}^n(B)$ and $\operatorname{Vect}_{\mathbb{C}}^n(B)$]
The set of isomorphism classes of $n$-dimensional \textit{real} vector bundles over $B$ is denoted by $\operatorname{Vect}^n(B)$\footnote{As a convention, we usually omit $\mathbb{R}$ as it is a common default assumption in math.}. Similarly, the set of isomorphism classes of $n$-dimensional \textit{complex} vector bundles over $B$ is denoted by $\operatorname{Vect}_{\mathbb{C}}^n(B)$. 
\end{definition}

Now let's use the collection of vector bundles over a base space to explore how a continuous map between two spaces can relate their corresponding vector bundles. This leads us to the concept of pullback bundles.

\begin{definition}[Pullback Bundle]\label{def:pullback_bundle}
Given a map $f: A \rightarrow B$ and a vector bundle $p: E \rightarrow B$, there exists a \textit{unique} (up to isomorphism) vector bundle $p^\prime: E^\prime \rightarrow A$ with a map $f^\prime: E^\prime \rightarrow E$ such that the fiber of $E^\prime$ over each point $a \in A$ is isomorphic onto the fiber of $E$ over $f(a)$. This vector bundle $E^\prime$ is called the \textbf{pullback} of $E$ by $f$ and is often denoted as $f^*(E)$.
\end{definition}

Notice that, in this definition, we assumed the existence and uniqueness of the vector bundles $p^\prime$ and $E'$. In fact, this can be easily proven as shown in Section 1.2 of \cite{hatcher2003vector}. The pullback bundle allows us to study how the structure of vector bundles is preserved or transformed under continuous maps between base spaces. By considering the uniqueness of the pullback bundle, we can establish a function $f^*: \operatorname{Vect}(B) \rightarrow \operatorname{Vect}(A)$, taking the isomorphism class of $E$ to the isomorphism class of $E^\prime$. This function gives us a way to relate the vector bundles of two spaces through a continuous map, and it further enriches our understanding of the interplay between topology and vector bundles.

\subsection{Two Versions of K-theory: \texorpdfstring{$\widetilde K(X)$}{K(X)} and \texorpdfstring{${K}(X)$}{K(X)}}

Building upon vector bundles, K-theory, which was developed by Michael Atiyah and Friedrich Hirzebruch in the late 1950s, is a powerful tool in algebraic topology and has since found applications in various branches of mathematics. 

The reduced K-theory group $\widetilde{K}(X)$ is introduced to better understand and organize the structure of vector bundles over a fixed base space $X$. To define it, we first need to introduce a related concept and establish some inspiration.

\begin{definition}[Stable Isomorphism]
Two vector bundles $E_1$ and $E_2$ over $X$ are called \textbf{stably isomorphic}, written $E_1 \approx_s E_2$, if $E_1 \oplus \varepsilon^n \approx E_2 \oplus \varepsilon^n$ for some $n$, where $\varepsilon^n$ is the trivial $n$-dimensional vector bundle over $X$.
\end{definition}

The motivation for this definition is to identify bundles that differ only by a \textit{trivial} bundle, focusing on their \textit{essential} structure. This idea of identification is widely used in math; for example, we identify spaces that are homotopy equivalent or homeomorphic, thereby focusing on their \textit{intrinsic} topological properties rather than the \textit{specific} details of their construction.

Now, let's consider the set of stable isomorphism classes of vector bundles over $X$. We can define an abelian group structure on this set, as shown in the following proposition:

\begin{proposition} 
If $X$ is compact Hausdorff, then the set of stable isomorphism classes of vector bundles over $X$ forms an abelian group with respect to the direct sum operation $\oplus$.
\end{proposition}

The motivation for this proposition is to organize vector bundles over a fixed base space $X$ into a group structure, which allows us to study their properties and interactions more systematically. Its proof relies on showing that the set of stable isomorphism classes of vector bundles satisfies the group axioms (closure, associativity, existence of identity, and existence of inverses) under the direct sum operation. 

Once we have established this fact, we can now introduce the reduced K-theory group $\widetilde{K}(X)$.

\begin{definition}[Reduced K-theory Group]
Given a compact Hausdorff space $X$, the abelian group formed by the set of stable isomorphism classes of vector bundles over $X$ under the direct sum operation $\oplus$ is called the \textbf{reduced K-theory group}, denoted by $\widetilde{K}(X)$.
\end{definition}

Let us provide an example for a better understanding of these definitions above.

\begin{example} 
Let $X =: S^1$, the circle. Consider the following two complex line bundles over $X$: the trivial bundle $E_1 = \varepsilon^1$ and the Möbius bundle $E_2$. By Example \ref{ex:line_bundle}, the Möbius bundle $E_2$ is not isomorphic to the trivial bundle $E_1$. However, we can show that $E_2 \oplus E_2 \approx \varepsilon^1$, which means, by definition, that $E_2$ and the trivial bundle $\varepsilon^1$ are \textit{stably isomorphic} and belong to the same stable isomorphism class in $\widetilde{K}(X)$. 
\end{example}

By defining the reduced K-theory group, we create a structure that allows us to analyze vector bundles in a systematic way. This group helps us uncover relationships between vector bundles and other algebraic and topological concepts, paving the way for motivation of the following definition. 


\begin{definition}[K-theory Group $K(X)$]
The \textbf{K-theory group} $K(X)$ consists of formal differences $E - E'$ of complex vector bundles $E$ and $E'$ over $X$, modulo an equivalence relation defined by "$E_1 - E_1' \sim E_2 - E_2'$ if and only if $E_1 \oplus E_2' \approx_s E_2 \oplus E_1'$". The group operation is given by the direct sum of vector bundles, i.e., $(E_1 - E_1') \oplus (E_2 - E_2') = (E_1 \oplus E_2) - (E_1' \oplus E_2')$. 
\end{definition}

The motivation for this definition is to create a more refined group structure that captures differences between vector bundles.

\begin{example} [K-theory Group of a Single Point] \label{ex:K_group_pt} 
Consider the base space $X$ to be a single point, denoted as $pt$. We will explore the K-theory group $K(pt)$ for this space.

Since $X$ consists of a single point, the only vector bundles over $X$ are trivial bundles. Let $n$ and $m$ be non-negative integers representing the dimensions of the trivial bundles $\varepsilon^n \rightarrow pt$ and $\varepsilon^m \rightarrow pt$, respectively. Notice that the direct sum of these bundles is another trivial bundle of dimension $n+m$, i.e., $\varepsilon^n \oplus \varepsilon^m \approx \varepsilon^{n+m}$.

By the definition above, the K-theory group $K(pt)$ consists of formal differences of trivial bundles, such as $\varepsilon^n - \varepsilon^m$. We can see that $\varepsilon^n - \varepsilon^m \sim \varepsilon^k - \varepsilon^l$ if and only if $\varepsilon^n \oplus \varepsilon^l \approx_s \varepsilon^k \oplus \varepsilon^m$. Since all the bundles are trivial, it follows that $\varepsilon^n \oplus \varepsilon^l \approx \varepsilon^{n+l}$ and $\varepsilon^k \oplus \varepsilon^m \approx \varepsilon^{k+m}$. Thus, the equivalence relation simply reduces to $n+l = k+m$. In addition, the group operation in $K(pt)$ is given by the direct sum of the differences of trivial bundles, i.e., $(\varepsilon^n - \varepsilon^m) \oplus (\varepsilon^k - \varepsilon^l) = (\varepsilon^n \oplus \varepsilon^k) - (\varepsilon^m \oplus \varepsilon^l) = \varepsilon^{n+k} - \varepsilon^{m+l}$.

From these observations, we can see that $K(pt)$ is isomorphic to the group of integers $\mathbb{Z}$ under the operation of addition, with the isomorphism given by $\varepsilon^n - \varepsilon^m \mapsto n - m$. 

\end{example}

Again, the motivation behind the reduced K-theory group $\widetilde{K}(X)$ and the K-theory group $K(X)$ is to provide a systematic way of studying vector bundles and their interactions. By organizing vector bundles into these groups, we are enabled to uncover deep results and explore various applications in algebraic topology, geometry, and mathematical physics.

\subsection{The Fundamental Product Theorem}\label{sec:Fundamental_Product_Theorem}

In Example \ref{ex:K_group_pt}, we have gained an in-depth understanding of K-theory groups of a single \textit{point}. In this subsection, we shall focus on the Fundamental Product Theorem, a key result that allows us to compute $K(X)$ for \textit{nontrivial} cases. Specifically, we aim to derive a formula for calculating $K\left(X \times S^{2}\right)$ in terms of $K(X)$. This formula will be crucial for deducing Bott Periodicity in the subsequent section.

Recall the canonical line bundle $H$ over $S^{2} $ is isomorphic to $ \mathbb{C}P^{1}$. It is easy to show the canonical line bundle $H\rightarrow \mathbb{C}P^{1}$ satisfies the relation $(H \otimes H) \oplus 1 \approx H \oplus H$. (Details see Example 1.13 in \cite{hatcher2003vector}). In $K\left(S^{2}\right)$, this translates to the formula $H^{2} + 1 = 2H$, i.e., $H^{2} = 2H - 1$, or, $(H-1)^{2} = 0$. Consequently, we have a natural ring homomorphism $\mathbb{Z}[H] /(H-1)^{2} \rightarrow K\left(S^{2}\right)$, with the domain being the quotient ring of the polynomial ring $\mathbb{Z}[H]$ by the ideal generated by $(H-1)^{2}$. Note that an additive basis for $\mathbb{Z}[H] /(H-1)^{2}$ is $\{1, H\}$. 

\begin{definition}[External Product]
The \textbf{external product} is a map that combines the $K$-theory groups of two spaces into a single $K$-theory group of their Cartesian product. It is defined as
\begin{equation*}
K(X) \otimes K(Y) \rightarrow K(X \times Y).
\end{equation*}
\end{definition}

With the external product in place, we can define a homomorphism $\mu$ as the composition:
\begin{equation*}
\mu: K(X) \otimes \mathbb{Z}[H] /(H-1)^{2} \rightarrow K(X) \otimes K\left(S^{2}\right) \rightarrow K\left(X \times S^{2}\right),
\end{equation*} where the second map is the external product. This homomorphism indeed can be further shown to be an isomorphism:

\begin{theorem}[Fundamental Product Theorem \cite{hatcher2003vector}]\label{th:fpt}
The homomorphism $\mu: K(X) \otimes \mathbb{Z}[H] /(H-1)^{2} \rightarrow K\left(X \times S^{2}\right)$ is an isomorphism of rings for all compact Hausdorff spaces $X$.
\end{theorem}

By taking $X$ as a point, we obtain the following corollary:

\begin{corollary}\label{cor:fpt}
The map $\mathbb{Z}[H] /(H-1)^{2} \rightarrow K\left(S^{2}\right)$ is an isomorphism of rings.
\end{corollary}

Thus, when we regard $\widetilde{K}\left(S^{2}\right)$ as the kernel of $K\left(S^{2}\right) \rightarrow K\left(x_{0}\right)$, it is generated as an abelian group by $H-1$. Since we have the relation $(H-1)^{2}=0$, this means that the multiplication in $\widetilde{K}\left(S^{2}\right)$ is completely trivial: the product of any two elements is zero.\footnote{Here the situation is similar to that in $H^{}\left(S^{2} ; \mathbb{Z}\right)$ and $\widetilde{H}^{}\left(S^{2} ; \mathbb{Z}\right)$, with $H-1$ behaving like the generator of $H^{2}\left(S^{2} ; \mathbb{Z}\right)$.}

\section{Bott Periodicity for Topological K-theory}

In this section, we delve into the intricate world of the first version of Bott Periodicity: Bott Periodicity for Topological K-theory, and explore its mathematical foundations, properties, theorems, and applications. 

\subsection{Exact Sequences and Splitness}

One of the key properties of the reduced groups $\tilde{K}(X)$ is the existence of exact sequences. Exact sequences are central to understanding the structure of the groups and will be useful in proving the periodicity results in $\mathrm{K}$-theory. We begin with the most general definition of exact sequences.

\begin{definition}[Exact Sequence]
An \textbf{exact sequence} is a sequence of morphisms in an abelian category (whose objects are groups, rings, modules, and so on) where the image of one morphism is equal to the kernel of the next.
\end{definition}

In particular, a \textbf{short exact sequence} is an exact sequence of the form
\[0 \longrightarrow A \stackrel{f}{\longrightarrow} B \stackrel{g}{\longrightarrow} C \longrightarrow 0,\]
where $f: A \rightarrow B$ and $g: B \rightarrow C$. Here our universe is the \textit{abelian category of abelian groups}. By definition, the short sequence is \textit{exact} if $f$ is injective, $g$ is surjective, and $\operatorname{im}(f) = \operatorname{ker}(g)$. We also call this short exact sequence an \textbf{extension} \textit{of} \(C\) \textit{by} \(A\).\footnote{See the following example for the explanation of another convention "extension \textit{of} \(A\) \textit{by} \(C\)".} Moreover, if  \(A \subset Z(B)\), namely \(A\) is included in the center of the group \(B\), we call this sequence a \textbf{central extension}. Let's use the following example about the second K-theory group \(K_{2}(R)\) of \textit{rings}\footnote{Notice in Section 1, the K-theory group we defined is for a compact \textit{Hausdorff} (topological) space $X$ that doesn't necessarily have a \textit{ring} structure.} in algebraic geometry to better understand the intuition behind central extensions.

\begin{example}[Second K-theory Group of Rings]
We define $K_{2}(R):=\operatorname{ker}(\varphi: {St}(R) \rightarrow   GL(R)  )$ for a ring $R$, called the \textbf{second K-theory group}. Here $St(R)$, called the \textbf{(stable) Steinberg group} of $R$ , is defined as the direct limit of  $St_{n}\left( R\right) := \langle x_{ij}\left( r\right) ,1<k,j\leqslant n\textbf{ }|\textbf{ }\mathbf {R}\rangle$, where $r \in R$, and the bold \textbf{R} refers to the Steinberg relations. By definition, we can show $K_{2}(R)=Z(S t(R))$, i.e., $K_{2}(R)$ is precisely the center of $S t(R)$. (See \cite{feng2021zeta} for more details and examples of this group.)

First, in particular, we obtain $K_{2}(R)$ is always an abelian group. Also, by definition, $1 \rightarrow K_{2}(R) \rightarrow S t(R) \rightarrow E(R) \rightarrow 1 $ is a \textit{central extension} of $ E(R)$ by $K_2(R)$; shortly, we say $St(R)$ or $S t(R) \rightarrow E(R)$ (not the whole exact sequence) is a central extension. Here “central'', of course, means the subgroup $K_2(R)$ lies in the center of $St(R)$. As for the phrase “extension of $E(R)$ by $K_2(R)$'', in fact, the idea of “central extension'' is partially motivated and regularly used in considering homotopy classes in homotopical algebra, and it's natural to say a central extension is \textit{of} the “base'' \textit{by} the “fiber''. As a result, we sometimes extend this idea to all extensions in general; that is, in the short exact sequence $0 \rightarrow A \stackrel{f}{\rightarrow} B \stackrel{g}{\rightarrow} C \rightarrow 0$, we may call this sequence an extension \textit{of} $C$ \textit{by} $A$ as defined above. However, the same sequence can also be called an extension \textit{of} $A$ \textit{by} $C$ in the literature, for example, in \cite{rotman2010advanced}. This latter naming is motivated and based on different origins, such as Galois theory.\footnote{Intuitively, in this example of the second K-group, the map $St(R) \to E(R)$ is surjective, namely $St(R)$ is “bigger'' than $E(R)$, so we can call $St(R)$ an extension of $E(R)$. On the other hand, $K_{2}(R) \rightarrow S t(R)$ is injective, so it also makes sense when we call $St(R)$ an extension of $K_{2}(R)$.}

\end{example}

Furthermore, we can give short exact sequences another structure:
\begin{definition} [Splitness]
A short exact sequence
$$
0 \rightarrow A \stackrel{f}{\longrightarrow} B \stackrel{g}{\longrightarrow} C \rightarrow 0
$$
is \textbf{split} if there exists a map $h: C \rightarrow B$ with $gh=1_C$
\end{definition}
We have the following direct consequence from abstract algebra:
\begin{proposition}\label{p.split}
If an exact sequence
$$
0 \rightarrow A \stackrel{f}{\longrightarrow} B \stackrel{g}{\longrightarrow} C \rightarrow 0
$$
is split, then $B \approx A \oplus C$.
\end{proposition} 

\subsection{Long Exact Sequences in K-theory}
Motivated by the example in the last subsection and the general notion of a sequence being exact, it's natural to form the following definition of exact sequences of homomorphisms as a particular case:

\begin{definition}[Exact Sequence of Homomorphisms]
A sequence of homomorphisms $G_{1} \rightarrow G_{2} \rightarrow \cdots \rightarrow G_{n}$ is said to be \textbf{exact} if at each intermediate group $G_{i}$, the kernel of the outgoing map equals the image of the incoming map.
\end{definition}

With this definition in mind, now it's time to dive into the following main property:

\begin{proposition}[Exact Sequence Theorem for Compact Hausdorff Spaces] \label{P2.9}
If $X$ is compact Hausdorff and $A \subset X$ is a closed subspace, then the inclusion and quotient maps $A \stackrel{i}{\longrightarrow} X \stackrel{q}{\longrightarrow} X / A$ induce an \textit{exact} sequence of homomorphisms $\tilde{K}(X / A) \stackrel{q^{*}}{\longrightarrow}$ $\tilde{K}(X) \stackrel{i^{*}}{\longrightarrow} \widetilde{K}(A)$. That is, the inclusion map $i: A \hookrightarrow X$ and quotient map $q: X \to X/A$ induce homomorphisms $q^*: \tilde{K}(X/A) \to \tilde{K}(X)$ and $i^*: \tilde{K}(X) \to \tilde{K}(A)$ such that $\mathrm{Ker}(i^*) = \mathrm{Im}(q^*)$. 
\end{proposition}

The motivation behind this proposition is to relate the $\mathrm{K}$-theory of $X$, $A$, and $X / A$ through exact sequences. This will help us understand how $\mathrm{K}$-theory behaves under various operations on topological spaces. Below is a sketch of the proof; details can be found in \cite{hatcher2003vector}.

\begin{proof}
Naturally, to show two sets are equal, we need to prove the two directions: showing that $\operatorname{Im} q^{*} \subset \operatorname{Ker} i^{*}$ and $\operatorname{Ker} i^{*} \subset \operatorname{Im} q^{*}$.

For the first part (easier direction), we need to show that $i^{*} q^{*}=0$. This follows directly from the fact that $q i$ is equal to the composition $A \rightarrow A / A \hookrightarrow X / A$, and $\tilde{K}(A / A)=0$.

The second part (harder direction) involves showing that if the restriction over $A$ of a vector bundle $p: E \rightarrow X$ is stably trivial, then there exists a vector bundle over $X / A$ such that $q^{}(E / h) \approx E$. This is done by constructing a quotient space $E / h$ of $E$ under certain identifications and proving that it is a vector bundle. The key observation is that $E$ is trivial over a neighborhood of $A$, and we can use this fact to construct local trivializations over a neighborhood of the point $A / A$. Finally, we easily verify that $E \approx q^{*}(E / h)$.
\end{proof}

Building upon Proposition \ref{P2.9}, we now aim to extend the exact sequence $\tilde{K}(X / A) \rightarrow \tilde{K}(X) \rightarrow \widetilde{K}(A)$ to the left, using a diagram involving cones and suspensions denoted by $C$ and $S$. This extension will provide a deeper understanding of the interrelationships between the $\tilde{K}$ groups of different spaces:
\[\begin{array}{ccc}
A \hookrightarrow X \hookrightarrow X \cup C A \hookrightarrow(X \cup C A) \cup C X \hookrightarrow((X \cup C A) \cup C X) \cup C(X \cup C A) \\

\end{array}\]
In this sequence, each space is obtained from its predecessor by attaching a cone on the subspace two steps back in the sequence. This process illustrates how the spaces are related to each other, with the cone construction allowing us to capture local information about the spaces. 

Moreover, for the last three spaces we have the following quotient maps:
\[\begin{cases}X \cup C A\rightarrow X/A\\ (X \cup C A) \cup C X\rightarrow SA\\ ((X \cup C A) \cup C X) \cup C(X \cup C A)\rightarrow SX\end{cases}\]
These quotient maps are obtained by collapsing the most recently attached cone to a point. Oftentimes, the quotient map collapsing a contractible subspace to a point is a homotopy equivalence, thus inducing an isomorphism on $\tilde{K}$. This observation leads us to the following lemma, which will be a key ingredient in obtaining a long exact sequence of $\tilde{K}$ groups:

\begin{lemma}[Cone Quotient Lemma]\label{l2.10}
Let $A$ be a contractible subspace of $X$. If $q: X \rightarrow X / A$ is the quotient map obtained by collapsing $A$ to a point, then the induced map  $q^{*}: \operatorname{Vect}^{n}(X / A) \rightarrow \operatorname{Vect}^{n}(X)$ is a bijection for all $n \geq 0$.
\end{lemma}

The Cone Quotient Lemma illustrates the importance of understanding how contractible subspaces relate to the sequence above. By collapsing a contractible subspace to a point, we can obtain useful information about the topological properties of the spaces and their associated vector bundles.

\begin{proof}
The idea of the proof is as follows. We first consider a vector bundle $E \rightarrow X$. It must be trivial over $A$ since $A$ is contractible. By using a trivialization $h$, we can construct a vector bundle $E / h \rightarrow X / A$ and demonstrate that the isomorphism class of $E / h$ does not depend on $h$. Finally, we establish a well-defined map between the vector bundle spaces and show that it is an inverse to $q^{*}$.
\end{proof}

In conclusion, let's combine the results of Lemma \ref{l2.10} and Proposition \ref{P2.9}. That gives us the main theorem in this subsection. 

\begin{theorem}[Long Exact Sequence in K-theory]\label{les}
Let $X$ be a topological space, and let $A \subset X$ be a closed subspace. Then there exists a long exact sequence of K-theory groups:
\begin{equation*}
    \cdots \rightarrow \widetilde{K}(S X) \rightarrow \widetilde{K}(S A) \rightarrow \widetilde{K}(X / A) \rightarrow \widetilde{K}(X) \rightarrow \widetilde{K}(A) \rightarrow \cdots
\end{equation*}
\end{theorem}

This theorem is significant in the sense that it provides a way to compute the K-theory groups of more complicated spaces, such as spheres, which is essential for the development of the Bott Periodicity Theorem in the next subsection.

\begin{example}\label{e:wedge_sum}
In particular,  let $X:=A \vee B$, be the \textbf{wedge sum} (or wedge product) of the spaces \(A\) and \(B\). That is, $X$ is formed by identifying a basepoint from $A$ and a basepoint from $B$. 
When we consider the quotient space $X / A$, we are essentially collapsing the entire subspace $A$ to a single point. Since $X$ is formed by identifying a basepoint in $A$ with a basepoint in $B$, collapsing $A$ to a single point also identifies the basepoint of $B$ with this collapsed point. The resulting space is (homeomorphic to) $B$ because it retains the same structure as $B$, except that its basepoint has been identified with the collapsed point. Thus, we have $X / A = B$.

In this case, the sequence in Theorem \ref{les} breaks up into split short exact sequences. By Proposition \ref{p.split}, this implies that we have the following isomorphism $$\widetilde{K}(A \vee B) \approx \widetilde{K}(A) \oplus \widetilde{K}(B)$$ obtained by restriction to $A$ and $B$. This result demonstrates the close relationship between the $\tilde{K}$ groups of $A$, $B$, and $X$ and highlights the additivity property of K-theory for wedge sums.
\end{example}

\subsection{Deducing Periodicity from Exact Sequences}

In this subsection, we aim to finally deduce periodicity properties of the K-theory using the Product Theorem as discussed Section \ref{sec:Fundamental_Product_Theorem}. To achieve this, we will utilize the smash product and reduced external product, ultimately leading us to the Bott Periodicity Theorem.

\begin{definition}[Smash Product]
Given topological spaces $X$ and $Y$ with chosen basepoints $x_{0} \in X$ and $y_{0} \in Y$, the \textbf{smash product} of $X$ and $Y$, denoted $X \wedge Y$, is the quotient space $X \times Y / X \vee Y$, where $X \vee Y := X \times \{y_{0}\} \cup \{x_{0}\} \times Y \subset X \times Y$ is the wedge sum for chosen basepoints $x_0 \in X$ and $y_0 \in Y$ as discussed in Example \ref{e:wedge_sum}.
\end{definition}

The smash product \textit{combines} properties of both $X$ and $Y$, while simplifying their topology by identifying certain parts of the product. This construction will help us in formulating the reduced external product.

\begin{definition}[Reduced Suspension]
Given a space $Z$ with basepoint $z_{0}$, the \textbf{reduced suspension} of $Z$, denoted $\Sigma Z$, is the quotient space of the suspension $SZ$ obtained by collapsing the segment $\{z_{0}\} \times I$ to a point, where $I:= [0, 1]$ is the unit interval.
\end{definition}

The reduced suspension is also a functional construction that preserves many properties of the original space while simplifying the topology. For this concept, we have the following natural (linear) property:

\begin{proposition}[Distributive Property of Reduced Suspension]\label{prop:distributive_reduced_suspension}
Given compact Hausdorff spaces $X$ and $Y$, the reduced suspension distributes over the wedge sum:
\[\Sigma (X\vee Y)=\Sigma X\vee \Sigma Y.\]
\end{proposition}

\begin{proof}
This property can be intuitively understood by definition. The reduced suspension can be viewed as a quotient of the ordinary suspension, collapsing an interval to a point. When applied to the wedge sum, the intervals corresponding to the base points of $X$ and $Y$ are identified, resulting in the wedge sum of the reduced suspensions.
\end{proof}

Now, let us consider the long exact sequence for the pair $(X \times Y, X \vee Y)$:
\begin{equation}\label{eq:long_exact_sequence}
\begin{gathered}
\widetilde{K}(S(X \times Y)) \longrightarrow \widetilde{K}(S(X \vee Y)) \longrightarrow \widetilde{K}(X \wedge Y) \longrightarrow \widetilde{K}(X \times Y) \longrightarrow \widetilde{K}(X \vee Y) 
\end{gathered}
\end{equation}
We can easily observe:
\begin{proposition}\label{th: iso}
For compact Hausdorff spaces $X$ and $Y$ as shown in the sequence above, we have the following isomorphisms:
\begin{enumerate}
    \item (Second) $\widetilde{K}(S(X \vee Y)) \approx \widetilde{K}(S X) \oplus \widetilde{K}(S Y)$.
    \item (Last) $\widetilde{K}(X \vee Y) \approx \widetilde{K}(X) \oplus \widetilde{K}(Y)$.
\end{enumerate}
\end{proposition}

\begin{proof}
2. The last isomorphism is a direct consequence of the Example \ref{e:wedge_sum} in the previous subsection.

1. To prove the first isomorphism, we begin by the reduced suspension distributes over the wedge sum \(\Sigma (X\vee Y)=\Sigma X\vee \Sigma Y\) by Proposition \ref{prop:distributive_reduced_suspension}.
Now, let us consider the K-theory of both sides of this equation; we have
\[\widetilde{K}(\Sigma (X\vee Y)) = \widetilde{K}(\Sigma X\vee \Sigma Y). \]
Applying the K-theory direct sum property ("the last isomorphism" proven above) to the right-hand side, we obtain:
\[\widetilde{K}(\Sigma X\vee \Sigma Y) = \widetilde{K}(\Sigma X) \oplus \widetilde{K}(\Sigma Y). \]
Thus 
\[\widetilde{K}(\Sigma (X\vee Y)) = \widetilde{K}(\Sigma X) \oplus \widetilde{K}(\Sigma Y).\]

By Lemma \ref{l2.10}, we know that the quotient map \(q: \Sigma Z\rightarrow SZ\) induces an isomorphism \(\widetilde{K}(\Sigma Z)\approx \widetilde{K}(SZ)\) for \(Z=X, Y, X\vee Y\). 
In other words, the reduced suspensions induces isomorphisms:
$$\begin{cases}
    \widetilde{K}(\Sigma X)\approx \widetilde{K}(SX)\\
\widetilde{K}(\Sigma Y)\approx \widetilde{K}(SY)\\
\widetilde{K}(\Sigma (X\vee Y)) \approx \widetilde{K}(S(X\vee Y)). 
\end{cases}$$

Hence, we have:

\[\widetilde{K}(S(X\vee Y)) \approx \widetilde{K}(S X) \oplus \widetilde{K}(S Y). \]

This establishes the first isomorphism in this proposition.
\end{proof}

Building upon the previous long exact sequence \eqref{eq:long_exact_sequence} and Proposition \ref{th: iso}, we can easily obtain the following facts:

\begin{proposition}\label{th: split_surjection}
For compact Hausdorff spaces $X$ and $Y$, the last map in the long exact sequence 
for the pair $(X \times Y, X \vee Y)$ is a \textit{split} surjection with splitting $$\widetilde{K}(X) \oplus \widetilde{K}(Y)\rightarrow \widetilde{K}(X \times Y).$$
\end{proposition}

\begin{proof}
By Proposition \ref{th: iso} (Example \ref{e:wedge_sum}), we have established the isomorphism:
\[\widetilde{K}(X \vee Y) \approx \widetilde{K}(X) \oplus \widetilde{K}(Y). \]
Since the last map in the long exact sequence \eqref{eq:long_exact_sequence} is a surjection, there exists a splitting map from $\widetilde{K}(X) \oplus \widetilde{K}(Y)$ to $\widetilde{K}(X \times Y)$ given by $(a, b) \mapsto p_{1}^{*}(a)+p_{2}^{*}(b)$, where $p_{1}$ and $p_{2}$ are the projections of $X \times Y$ onto $X$ and $Y$, respectively. 
\end{proof}

\begin{proposition}\label{th: first_map_split}
For compact Hausdorff spaces $X$ and $Y$, the first map in the long exact sequence for the pair $(X \times Y, X \vee Y)$ \textit{splits}. As a result, we obtain a splitting $$\tilde{K}(X \times Y) \approx \widetilde{K}(X \wedge Y) \oplus \tilde{K}(X) \oplus \tilde{K}(Y).$$
\end{proposition}

\begin{proof}
Continuing from the last proof of Proposition \ref{th: split_surjection}, we know that the first map in the long exact sequence splits via $\left(S p_{1}\right)^{*}+\left(S p_{2}\right)^{*}$. By Proposition \ref{p.split},
this splitting implies that $\tilde{K}(X \times Y)$ is isomorphic to the direct sum $\widetilde{K}(X \wedge Y) \oplus \tilde{K}(X) \oplus \tilde{K}(Y)$. 
\end{proof}

\begin{proposition}\label{th: external_product}
For $a \in \widetilde{K}(X)$ and $b \in \widetilde{K}(Y)$, the external product $a * b$ lies in $\tilde{K}(X \times Y)$ and induces a unique element in $\widetilde{K}(X \wedge Y)$, defining a reduced external product $\tilde{K}(X) \otimes \tilde{K}(Y) \rightarrow \widetilde{K}(X \wedge Y)$.
\end{proposition}

\begin{proof}
Recall that $p_1^{*}(a)$ restricts to zero in $K(Y)$ and $p_2^{}(b)$ restricts to zero in $K(X)$. As a result, $p_1^{*}(a) p_2^{*}(b)$ restricts to zero in both $K(X)$ and $K(Y)$, and therefore in $K(X \vee Y)$. Consequently, $a * b$ lies in $\tilde{K}(X \times Y)$, and from the short exact sequence given by Proposition \ref{th: first_map_split}, $a * b$ pulls back to a unique element of $\widetilde{K}(X \wedge Y)$. This defines the reduced external product.
\end{proof}

Now combining these propositions above, we obtain the following commutative diagram:
\[\begin{array}{ccc} K(X) \otimes K(Y) & \approx(\widetilde{K}(X) \otimes \widetilde{K}(Y)) \oplus \widetilde{K}(X) \oplus \widetilde{K}(Y) \oplus \mathbb{Z} \\ \downarrow & \downarrow  \\ K(X \times Y) & \approx \widetilde{K}(X \wedge Y) \oplus \widetilde{K}(X) \oplus \widetilde{K}(Y) \oplus \mathbb{Z} \end{array}\]

Since the reduced external product is essentially a restriction of the unreduced external product, it also preserves the ring structure and is therefore also a ring homomorphism. With this in mind, let's use the notation $a * b$ for both reduced and unreduced external products.

From Proposition \ref{th: external_product}, we defined the reduced external product. Next, we will show that this reduced external product gives rise to the homomorphism $\beta$ as described in the following theorem.

\begin{lemma}\label{th: reduced_external_product_suspension}
The reduced external product gives rise to a homomorphism
\[
\beta: \tilde{K}(X) \rightarrow \widetilde{K}\left(S^{2} X\right), \quad \beta(a)=(H-1) * a,
\]
where $H$ is the canonical line bundle over $S^{2}=\mathbb{C P}^{1}$.
\end{lemma}

\begin{proof}

Consider the $n$-fold iterated reduced suspension $\Sigma^n X$ of the space $X$. This space is a quotient of the ordinary $n$-fold suspension $S^n X$, which is obtained by collapsing an $n$-disk in $S^n X$ to a point. Consequently, the quotient map $S^n X \rightarrow S^n \wedge X$ is defined. 
By Lemma \ref{l2.10}, we know that this quotient map $S^n X \rightarrow S^n \wedge X$ induces an isomorphism on $\tilde{K}$. In other words, we have an isomorphism between the reduced K-theories of $S^n X$ and $S^n \wedge X$.

Therefore, we can now define the homomorphism $\beta: \tilde{K}(X) \rightarrow \widetilde{K}\left(S^{2} X\right)$.\footnote{The motivation behind defining $\beta$ is to study the relationship between the K-theory of a space and its double suspension.} Specifically, we define $\beta(a):=(H-1) * a$, where $H$ is the canonical line bundle over $S^{2}=\mathbb{C P}^{1}$, and the operation $*$ denotes the reduced external product as defined in Proposition \ref{th: external_product}.
\end{proof}

So far, we have analyzed the reduced external product and its relation to the K-theory of these spaces, and shown that it leads to a homomorphism between the reduced K-theories of $X$ and the double suspension of $X$. As we did for deducing the fundamental product theorem, motivated by the previous lemma, it is natural to ask: wouldn't it be great if the \textit{homomorphism} \(\beta: \tilde{K}(X) \rightarrow \widetilde{K}\left(S^{2} X\right)\) was further an \textit{isomorphism}? If that is true, we will obtain a powerful tool to examine and compute reduced K groups since this isomorphism will exhibit a periodic property. This idea leads us to the discovery of Bott Periodicity Theorem, a central result in K-theory and one of our two main results in this paper:

\begin{theorem}[Bott Periodicity Theorem for Topological K-theory]
$\tilde{K}(X)\stackrel{\beta}{\approx}\tilde{K}\left(S^{2} X\right)$ by $\beta(a):=(H-1) * a $, $\forall$ compact Hausdorff $X$. 
That is, the homomorphism $\beta: \tilde{K}(X) \rightarrow \tilde{K}\left(S^{2} X\right)$, defined by $\beta(a)=(H-1) * a$ as in the Lemma above, is an isomorphism for all compact Hausdorff spaces $X$.
\end{theorem}

\begin{proof}
One way (a sinuous way) to show that $\beta:\widetilde{K}(X) {\rightarrow} \widetilde{K}\left(S^{2} X\right)$ is an isomorphism is to \textit{decompose} $\beta$  into two maps $\beta_1\beta_2$ and then show they are \textit{both} isomorphisms. That is, we write
\[\widetilde{K}(X) \stackrel{\beta_1}{\rightarrow} \widetilde{K}\left(S^{2}\right) \otimes \widetilde{K}(X) \stackrel{\beta_2}{\rightarrow} \widetilde{K}\left(S^{2} X\right).\]

The first map $\beta_1$ sends $a$ to $(H-1) \otimes a$, where $H$ is the canonical line bundle over $S^{2}=\mathbb{C P}^{1}$. This map is an isomorphism because $\tilde{K}\left(S^{2}\right)$ is infinite cyclic generated by $H-1$.

The second map $\beta_2$ is defined as the reduced external product. Since $\widetilde{K}\left(S^{2}\right)$ is generated by $H-1$ as an abelian group and has the relation $(H-1)^{2}=0$, we can rewrite the map $\beta_2$ as a composition:
\[\beta_2: \widetilde{K}(X) \otimes \widetilde{K}\left(S^{2}\right) \rightarrow K(X) \otimes \mathbb{Z}[H] /(H-1)^{2} \rightarrow K\left(X \times S^{2}\right) \rightarrow \widetilde{K}\left(S^{2} X\right).\]
The first map here is an isomorphism due to Corollary \ref{cor:fpt}. The second map here is the homomorphism $\mu$, which is an isomorphism according to Theorem \ref{th:fpt}. The third map here is the natural map induced by the quotient map $X \times S^{2} \rightarrow S^{2} X$ that collapses an $n$-disk in $X \times S^{2}$ to a point. This map induces an isomorphism on $\tilde{K}$. Since all the maps in the composition are isomorphisms, the composition $\beta_2$ is also an isomorphism. Thus, $\beta_2: \widetilde{K}\left(S^{2}\right) \otimes \widetilde{K}(X) \rightarrow \widetilde{K}\left(S^{2} X\right)$ is an isomorphism, as required. 

Therefore, we have concluded that the homomorphism $\beta: \tilde{K}(X) \rightarrow \tilde{K}\left(S^{2} X\right)$, defined by $\beta(a)=(H-1) * a$, is an isomorphism for all compact Hausdorff spaces $X$.
\end{proof}

The Bott Periodicity Theorem is so named due to the periodic behavior it reveals in K-theory. Specifically, it shows that there is a repeating pattern in the K-theory groups as we move through suspensions, with a period of two. In other words, as we see above, the K-theory of a space $X$ and its double suspension $S^2 X$ are isomorphic. This remarkable periodicity property plays a fundamental role in the study of K-theory and has far-reaching consequences in algebraic topology and other areas of mathematics.

This theorem is named after Raoul Bott (1923 – 2005), an American mathematician who first discovered this periodicity phenomenon in the late 1950s in \cite{bott1957stable} and \cite{bott1959stable}. His work was groundbreaking, as it provided a new and powerful tool for studying the homotopy groups of spheres and other topological spaces. Bott's discovery of this periodicity also led to deeper connections between K-theory and other areas of mathematics, such as representation theory and algebraic geometry. His work laid the foundation for subsequent developments in the field, and the Bott Periodicity Theorem remains a cornerstone of K-theory to this day. 
\begin{figure}[htbp]
  \centering
  \includegraphics[width=0.5\textwidth]{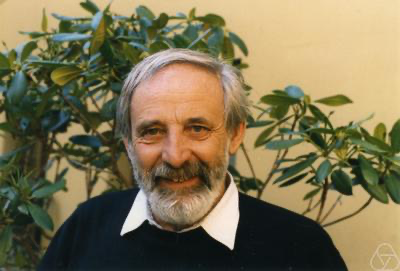}
  \caption{Raoul Bott\protect\footnotemark}
  \label{fig:raoul-bott}
\end{figure}
\footnotetext{Image source: 
\href{https://commons.wikimedia.org/wiki/File:Raoul_Bott_1986.jpeg}{Wikimedia Commons}, 
licensed under 
\href{https://www.gnu.org/licenses/old-licenses/fdl-1.2.html}{GFDL 1.2}.}

As a consequence of the Bott Periodicity Theorem, we obtain the following corollary, which gives us a deeper understanding of the periodicity properties of K-theory and highlights the significance of the smash product and reduced external product in connecting the K-theory of different spaces.

\begin{corollary}
$\widetilde{K}\left(S^{2 n+1}\right)=0$ and $\widetilde{K}\left(S^{2 n}\right) \approx \mathbb{Z}$, generated by the $n$-fold reduced external product $(H-1) * \cdots *(H-1)$.
\end{corollary}

Notice by this corollary of the Bott Periodicity Theorem, we have obtained a very generalized proposition: the K-theory of odd-dimensional spheres is trivial (i.e., consisting only of the zero element). As well, $\widetilde{K}\left(S^{2 n}\right) \approx \mathbb{Z}$ demonstrates that the K-theory of the \(2n\)-sphere exhibits a periodic structure. 

\begin{example}
Consider the case when $n = 1$. According to the corollary, we have: \(\widetilde{K}\left(S^{2}\right) \approx \mathbb{Z}\) (generated by the reduced external product $(H-1)$), and \(
\widetilde{K}\left(S^{3}\right) = 0\) (which is trivial since 3 is odd). That is, we have easily computed the reduced K-groups of the 2-sphere, $S^2$, and the 3-sphere, $S^3$ by this consequence of the Bott Periodicity Theorem.
\end{example}

\section{Bott Periodicity and Fiber Bundles}

In this section, we explore fiber bundles and their associated properties, ultimately discussing the second version Bott Periodicity: Bott Periodicity for stable homotopy groups of classical groups. This was originally discovered to prove periodicity of homotopy groups for the infinite unitary group.

\subsection{Fiber Bundles and Basic Properties}

At the beginning of the paper, we introduced vector bundles. In fact, a vector bundle is a specific type of a more generalized construction as shown in the following definition:

\begin{definition}[Fiber Bundle] 
A \textbf{fiber bundle} over $B$ is a space $E$ with a map $\pi: E \rightarrow B$ (continuous), satisfying "local triviality": 
$\forall b \in B$, denoting $E_b:=\pi^{-1}(b), \exists$ open $U \ni b$ in $B$ and a map $\left.E\right|_U:=\pi^{-1}(U) \stackrel{t}{\longrightarrow} E_b$ such that the map $E_{U} \stackrel{(\pi, t)=\varphi}{\longrightarrow} U \times E_{b} $ is a homeomorphism. Note that any two fibers of a fiber bundle in the same connected component of $B$ must be homeomorphic.
\end{definition}

In other words, a \textbf{fiber bundle} is a short exact sequence\footnote{In this context,  by Proposition \ref{p.split}, “exact” is used loosely as an intuitive way to capture the idea that locally, 
$E$ behaves like the product $B \times F$} of spaces $F \rightarrow E \stackrel {p}\longrightarrow B$, in which all the subspaces $p^{-1}(b) \subset E$, called \textbf{fibers}, are homeomorphic. The fiber bundle structure is characterized by the projection map $p: E \rightarrow B$. To denote the specific fiber, we also represent a fiber bundle using the notation $F \rightarrow E \rightarrow B$, which forms a \textit{short exact sequence of spaces}. The last space $B$ is called the \textbf{base space} of the bundle, while the middle space $E$ is called the \textbf{total space}. 

Comparing their definitions, we can see a \textit{vector bundle} is a \textit{fiber bundle} where the fibers are vector spaces, and the projection map is compatible with the vector space structure. Fiber bundles can also be thought of as twisted products. Familiar examples of fiber bundles are the \textit{Möbius band} (which is a twisted annulus with line segments as fibers) and the \textit{Klein bottle} (which is a twisted torus with circles as fibers). Fiber bundles provide a rich structure to study topological spaces and their relationships.

\begin{definition}[Homotopy Lifting Property] 
A map $p: E \rightarrow B$ is said to have the \textbf{homotopy lifting property} with respect to a space $X$ if, given a homotopy $g_t: X \rightarrow B$ and a map $\tilde{g}_0: X \rightarrow E$ lifting $g_0$, so that $p \tilde{g}_0 = g_0$, there exists a homotopy $\tilde{g}_t: X \rightarrow E$ lifting $g_t$.
\end{definition}

The homotopy lifting property plays a critical role in studying the relationship between spaces involved in a fiber bundle, as it enables the construction of homotopies between them:

\begin{theorem}[Homotopy Lifting Property Isomorphism Theorem]\label{thm:homotopy_lifting} 
Let $p: E \rightarrow B$ be a map satisfying the \textit{homotopy lifting property} with respect to disks $D^k$ for all $k \geq 0$. Choose basepoints $b_0 \in B$ and $x_0 \in F = p^{-1}(b_0)$. Then, we have the isomorphism $\pi_n(E, F, x_0) \stackrel{p_*}{\approx}\pi_n(B, b_0)$ by the induced map $p_*: \pi_n(E, F, x_0) \rightarrow \pi_n(B, b_0)$ for all $n \geq 1$. Consequently, if $B$ is path-connected, there exists a long exact sequence of homotopy groups:
$$
\cdots \rightarrow \pi_{n}\left(F, x_{0}\right) \rightarrow \pi_{n}\left(E, x_{0}\right) \stackrel{p_{*}}{\rightarrow} \pi_{n}\left(B, b_{0}\right) \rightarrow \pi_{n-1}\left(F, x_{0}\right) \rightarrow \cdots \rightarrow \pi_{0}\left(E, x_{0}\right) \rightarrow 0
$$
\end{theorem}

This theorem establishes the connection between the homotopy groups of the spaces involved in a fiber bundle when the homotopy lifting property is satisfied. The proof of the theorem can be found in Chapter 4 of \cite{hatcher2002algebraic}.

With the definition of the homotopy lifting property, we define the term \textbf{fibration} to be a map $p:E\rightarrow B$ that satisfies this homotopy lifting property with respect to all spaces \(X\). In particular, if the map $p$ satisfies the homotopy lifting property \textit{for disks}, then it's usually called a \textbf{Serre fibration}.

\subsection{Stiefel Manifolds in Real, Complex, and Quaternionic Cases}

\begin{definition}[Stiefel Manifold]
A \textbf{Stiefel manifold}, denoted by $V_{n}(\mathbb{R}^{k})$, or $V_{n,k}$, is a space consisting of all orthonormal $k$-frames in $\mathbb{R}^n$; in other words, it is the set of all orthonormal $k$-tuples of vectors in $\mathbb{R}^n$. It is given the subspace topology as a subset of the product of $k$ copies of the $(n-1)$-sphere, $S^{n-1}$.
\end{definition}

By this definition, $V_{n}(\mathbb{R}^{k})$ is a space of orthonormal $k$-tuples of vectors, which can be thought of as a subspace of the product of $k$ copies of the $(n-1)$-sphere. Oftentimes, the Stiefel manifold is used to study various properties of orthonormal frames, and its homotopy and cohomology groups are of particular interest in algebraic topology. In addition, we have the following property (as discussed in Ganatra's \cite{ganatra2023lecture17}):

\begin{proposition}[Compactness of $V_{n}(\mathbb{R}^{k})$]
The Stiefel manifold $V_{n}(\mathbb{R}^{k})$ is a compact manifold. 
\end{proposition}

\begin{proof}
To start, observe $O(n)$ acts on $V_k\left(\mathbb{R}^n\right)$ by composition: transitive action. Also, the isotropy group of basepoints $\left\{e_{1},..., e_k\right\}$ is $I_k \times O(n-k)$. Using this, we can easily show $V_k\left(\mathbb{R}^n\right)\approx O(n) / (I_k \times O{(n-k)})$. 
\end{proof}

As an example, let us consider the general case of fiber bundles for $m < n \leq k$:
\begin{equation}\label{eq:fb}
V_{n-m}\left(\mathbb{R}^{k-m}\right) \rightarrow V_{n}\left(\mathbb{R}^{k}\right) \stackrel{p}{\longrightarrow} V_{m}\left(\mathbb{R}^{k}\right)
\end{equation}
Here, we have three Stiefel manifolds involved in the fiber bundle. The \textit{base space} is $V_m(\mathbb{R}^k)$, which consists of all orthonormal $m$-frames in $\mathbb{R}^k$. The \textit{total space} is $V_n(\mathbb{R}^k)$, which consists of all orthonormal $n$-frames in $\mathbb{R}^k$. The fiber over a point in the base space is $V_{n-m}(\mathbb{R}^{k-m})$, which consists of all orthonormal $(n-m)$-frames in $\mathbb{R}^{k-m}$. In this fiber bundle, the projection map $p: V_{n}(\mathbb{R}^k) \rightarrow V_{m}(\mathbb{R}^k)$ takes an orthonormal $n$-frame in $\mathbb{R}^k$ and maps it onto an orthonormal $m$-frame in $\mathbb{R}^k$.

First, let's restrict \eqref{eq:fb} to the case $m=1$, so we have bundles \[V_{n-1}\left(\mathbb{R}^{k-1}\right) \rightarrow V_{n}\left(\mathbb{R}^{k}\right) \stackrel{p}{\longrightarrow} V_{1}\left(\mathbb{R}^{k}\right).\]Notice that $V_{1}\left(\mathbb{R}^{k}\right)$ can be identified with the $(k-1)$-dimensional sphere $S^{k-1}$. This is because the 1-frames in $\mathbb{R}^{k}$ are essentially the unit vectors in $\mathbb{R}^{k}$, and the unit vectors form the $(k-1)$-dimensional sphere $S^{k-1}$. Thus, we can rewrite the fiber bundle as  $$V_{n-1}\left(\mathbb{R}^{k-1}\right) \rightarrow V_{n}\left(\mathbb{R}^{k}\right) \stackrel{p}{\longrightarrow} S^{k-1}.$$ That allows us to deduce the connectivity of $V_{n}\left(\mathbb{R}^{k}\right)$ by induction on $n$.

Next, let's (only) restrict \eqref{eq:fb} to the case $k=n$, so we obtain fiber bundles \[V_{n-m}\left(\mathbb{R}^{n-m}\right) \rightarrow V_{n}\left(\mathbb{R}^{n}\right) \stackrel{p}{\longrightarrow} V_{m}\left(\mathbb{R}^{n}\right).\] Recall that $V_{n}\left(\mathbb{R}^{n}\right)$ represents the set of all orthonormal $n$-frames in $\mathbb{R}^{n}$, which can be identified with the orthogonal group $O(n)$. Likewise, the fiber $V_{n-m}\left(\mathbb{R}^{n-m}\right)$ can be identified with the orthogonal group $O(n-m)$ because it represents the set of orthonormal $(n-m)$-frames in $\mathbb{R}^{n-m}$. The projection map $p$ sends an $n$-frame in $O(n)$ to the $m$-frame formed by its first $m$ vectors, and the fibers consist of the cosets $\alpha O(n-m)$ for $\alpha \in O(n)$. Therefore, we can rewrite the fiber bundle as $$O(n-m) \rightarrow O(n) \stackrel{p}{\longrightarrow}  V_{m}\left(\mathbb{R}^{n}\right).$$

Now, let's combine the preceding two cases together: taking $m=1$ \textit{and} $k=n$ to obtain bundles
$$
O(n-1)  \rightarrow O(n) \stackrel{p}{\longrightarrow} S^{n-1}.
$$

Before we further explore these particular bundles, let's analogously define fiber bundles for the \textit{complex} and \textit{quaternionic} cases: 

\begin{definition}[Unitary and Symplectic Groups] 
The \textbf{unitary group} $U(n)$ is the group of $n \times n$ unitary matrices, i.e., complex matrices $U$ such that $U^* U = I$, where $U^*$ is the conjugate transpose of $U$. The \textbf{symplectic group} $Sp(n)$ consists of $2n \times 2n$ matrices $S$ that preserve a skew-symmetric bilinear form, i.e., $S^T JS = J$, where $J$ is a skew-symmetric matrix.
\end{definition}

To understand their fiber bundle structure, we first need to connect the \textit{Stiefel manifolds} to the coset spaces of these groups. The complex and quaternionic versions of Stiefel manifolds can be identified with the coset spaces of the unitary and symplectic groups, respectively. For example, $V_n(\mathbb{C}^k)$ is identifiable with the coset space $U(k) / U(k-n)$, and $V_n(\mathbb{H}^k)$ is identifiable with the coset space $Sp(k) / Sp(k-n)$. 

Now let's consider their corresponding Stiefel manifolds in the complex and quaternionic cases. Employing the same reasoning as above, by restricting to the case $m=1$ and $k=n$, we obtain fiber bundles in the complex and quaternionic cases as follows:
$$
\begin{aligned}
U(n-1) & \rightarrow U(n) \stackrel{p}{\longrightarrow} S^{2 n-1} \\
S p(n-1) & \rightarrow S p(n) \stackrel{p}{\longrightarrow} S^{4 n-1}
\end{aligned}
$$ 
These bundles are analogous to the real case, with the projection map $p$ in each case sending an element of $U(n)$ or $Sp(n)$ to the corresponding coset in the Stiefel manifold, which is homeomorphic to a sphere of dimension $2n-1$ or $4n-1$, respectively. Unfortunately, these bundles also show that computing homotopy groups of $O(n), U(n)$, and $S p(n)$ should be at least as difficult as computing homotopy groups of spheres.

\subsection{Deducing Periodicity from Fiber Bundles}
In the preceding subsection, we have examined the three bundles:
$$
\begin{aligned}
O(n-1) & \rightarrow O(n) \stackrel{p}{\longrightarrow} S^{n-1} \\
U(n-1) & \rightarrow U(n) \stackrel{p}{\longrightarrow} S^{2 n-1} \\
S p(n-1) & \rightarrow S p(n) \stackrel{p}{\longrightarrow} S^{4 n-1}
\end{aligned}
$$
They actually lead to an interesting stability property. In the real case, the inclusion $O(n-1) \hookrightarrow O(n)$ induces an isomorphism on $\pi_{i}$ for $i<n-2$ from the long exact sequence of the first bundle. As a consequence, the groups $\pi_{i} O(n)$ are independent of $n$ if $n$ is sufficiently large. Similarly, the same is true for the groups $\pi_{i} U(n)$ and $\pi_{i} S p(n)$ through the other two bundles. 

With these observations in mind, it is time to conclude one of the most remarkable results in algebraic topology, the \textit{Bott Periodicity Theorem for Stable Homotopy Groups of Classical Groups}, which states that these stable groups repeat periodically, with a period of eight for $O$ and $S p$, and a period of two for $U$. The theorem can be formally stated as follows:

\begin{theorem}[Bott Periodicity for Stable Homotopy Groups of Classical Groups]
The stable homotopy groups of $O(n)$, $U(n)$, and $S p(n)$ exhibit periodic behavior with periods of eight for $O$ and $S p$, and a period of two for $U$. The periodic values are given in the following table:

\[
\begin{tabular}{l|llllllll}
$i \bmod 8$ & \(0\) & \(1\) & \(2\) & \(3\) & \(4\) & \(5\) & \(6\) & \(7\) \\
\hline$\pi_{i} O(n)$ & $\mathbb{Z}_{2}$ & $\mathbb{Z}_{2}$ & \(0\) & $\mathbb{Z}$ & \(0\) & \(0\) & \(0\) & $\mathbb{Z}$ \\
$\pi_{i} U(n)$ & \(0\) & $\mathbb{Z}$ & \(0\) & $\mathbb{Z}$ & \(0\) & $\mathbb{Z}$ & \(0\) & $\mathbb{Z}$ \\
$\pi_{i} S p(n)$ & \(0\) & \(0\) & \(0\) & $\mathbb{Z}$ & $\mathbb{Z}_{2}$ & $\mathbb{Z}_{2}$ & \(0\) & $\mathbb{Z}$
\end{tabular}\]

\end{theorem}

This theorem not only highlights a deep connection between the homotopy groups of the orthogonal, unitary, and symplectic groups, but, most importantly, reveals an unexpected \textit{periodic} structure in these groups. As the Bott Periodicity Theorem for Topological K-theory, this version of Bott Periodicity Theorem also has far-reaching applications in algebraic topology, and it enables us to easily classify homotopy groups of $O(n)$, $U(n)$, and $S p(n)$.

\begin{example}
Consider the stable homotopy groups of the unitary group $U(n)$. According to the Bott Periodicity Theorem above, these groups have a period of two. This means that the homotopy groups $\pi_i U(n)$ and $\pi_{i+2} U(n)$ are isomorphic for all $i$. Namely, $\pi_i U(n) \approx \pi_{i+2} U(n)$ for all $i$. This shows a periodic behavior for the homotopy groups $\pi_i U(n)$ $\forall i$.
\end{example}

\section{Conclusion}
\label{sec:conclusion}

In the previous two sections, we have delved into two vibrant versions of Bott Periodicity that exhibit periodic behavior: the Bott Periodicity Theorem for \textit{topological K-theory}, and the Bott Periodicity Theorem for \textit{stable homotopy groups of classical groups} (orthogonal, unitary, and symplectic groups). The first version establishes an isomorphism between the reduced K-theories of a compact Hausdorff space and its double suspension, emphasizing the algebraic and geometric properties of vector bundles over topological spaces; whereas the second version reveals periodic patterns in stable homotopy groups, providing valuable insights into the global structure of these groups.

Although these two versions of Bott Periodicity focus on different aspects of topology, they are deeply interconnected, and both demonstrate the fascinating periodic structure in topology. They showcase the rich interplay between algebra, geometry, and topology.
The beauty and elegance of these periodic phenomena lie not only in their symmetric simplicity but also in the insights they provide into the boundless creativity that inspires the evolution of topology and mathematics in general, reaching far and wide across dimensions yet untold.

\end{document}